\documentclass{amsart}
\usepackage{amsthm,amsmath,amssymb,latexsym,soul,cite,mathrsfs}
\usepackage{color}
\usepackage{hyperref}
\usepackage[english,activeacute]{babel}
\usepackage{enumitem}

\definecolor{gris}{rgb}{0.5,0.3,0.3}

\usepackage[left=3.3cm,right=3.3cm,top=2.9cm,bottom=2.9cm]{geometry}

\newenvironment{acknowledgement}{\textbf{Acknowledgement.}\em}{}

\newtheorem{thm}{Theorem}[section]
\newtheorem{cor}[thm]{Corollary}
\newtheorem{lema}[thm]{Lemma}
\newtheorem{prop}[thm]{Proposition}
\theoremstyle{definition}

\theoremstyle{remark}

\newtheorem{rem}[thm]{Remark}
\numberwithin{equation}{section}
\newcommand{\R}{\mathbb R}
\newcommand{\N}{\mathbb N}

\newcommand{\LL}{\mathcal{L}_s}
\newcommand{\J}{\mathcal{J}_{s,G}}
\newcommand{\JJ}{\mathcal{J}_{1,\overline{G}}}

\newcommand{\ve}{\varepsilon}

\begin{document}
\title{Asymptotic behavior of Musielak-Orlicz-Sobolev modulars}

\author{J.C. de Albuquerque}
\address{J.C. de Albuquerque: Departamento de Matem\'atica, CCEN - Universidade Federal de Pernambuco, Av. Jorn. Aníbal Fernandes, S/N - Cidade Universitária,  Recife-PE, Brasil.}
\email{josecarlos.melojunior@ufpe.br}

\author{L.R.S. de Assis}
\address{L.R.S. de Assis: Departamento de Matem\'atica, CCEN - Universidade Federal de Pernambuco, Av. Jorn. An\'ibal Fernandes, S/N - Cidade Universitária,  Recife-PE, Brasil.}
\email{lazaro.assis@ufpe.br}

\author{M.L.M. Carvalho} 
\address{M.L.M. Carvalho: Departmento de Matem\'atica, Universidade Federal de Goi\'{a}s, 74001-970, Goi\'{a}s-GO, Brasil}
\email{marcos$\_$leandro$\_$carvalho@ufg.br}

\author{A. Salort}
\address{Ariel  Salort: Departamento de Matem\'atica, FCEyN - Universidad de Buenos Aires and IMAS - CONICET Ciudad Universitaria, Pabell\'on I (1428) Av. Cantilo s/n., Buenos Aires, Argentina.}
\email{asalort@dm.uba.ar}
\urladdr{http://mate.dm.uba.ar/~asalort}

\maketitle

\begin{abstract}
In this article we study the asymptotic behavior of anisotropic nonlocal nonstandard growth seminorms and modulars as the fractional parameter goes to 1. This gives a so-called Bourgain-Brezis-Mironescu type formula for a very general family of functionals. In the particu\-lar case of fractional Sobolev spaces with variable exponent, we point out that our proof asks for a weaker regularity of the exponent than the considered in previous articles.
\end{abstract}

\tableofcontents

\section{Introduction}

The analysis of the limit of the fractional parameter in fractional-order Sobolev type spaces has received some attention in the last years. The seminal work \cite{BBM}  paved the way to the development of an extensive literature related with the limit study of fractional parameters in several functional spaces.

In the seminal work of Bourgain-Brezis-Mironescu \cite{BBM}, the authors consider the classical fractional Sobolev spaces $W^{s,p}(\R^n)$, $s\in(0,1)$, $p\in [1,\infty)$ and study the behavior of the corres\-ponding Gagliardo-Slobodeckij seminorm as $s$ approaches 1. More precisely, they prove that, if $u\in W^{1,p}(\R^n)$, $p\in [1,\infty)$, then
$$
\lim_{s\uparrow 1} (1-s)\iint_{\R^n\times\R^n} \frac{|u(x)-u(y)|^p}{|x-y|^{n+sp}}\;dxdy = K(n,p)\int_{\R^n} |\nabla u|^p \,dx,
$$
where $K(n,p)=\frac1p \int_{\mathbb{S}^{n-1}}|\theta\cdot e|^p\,d \mathcal{H}^{n-1}(\theta)$, being $\mathbb{S}^{n-1}$ the $(n-1)$-dimensional unit sphere in $\R^n$, $\mathcal{H}^{n-1}$ the $(n-1)-$dimensional Hausdorff measure and $e\in \mathbb{S}^{n-1}$.

The case in which $\R^n$ is replaced by a bounded regular domain was considered in \cite{BBM} and \cite{davila}. The case of bounded extension domains was treated in \cite{BMR}, whilst \cite{DD} deals with arbitrary bounded domains. Similar results were proved to hold in more general fractional Sobolev spaces: the extension to the so-called magnetic fractional Sobolev spaces was dealt in \cite{PSV,SV}, and the case of spaces with anisotropic structure was studied in \cite{DB, FBS2}.

Later on, these kind of result were extended to a bigger class of functions allowing a behavior more general than a power. More precisely, the fractional Orlicz-Sobolev space $W^{s,G}(\R^n)$ is defined in terms of a Young function $G$, namely, a convex function from $[0,\infty)$ into $[0,\infty]$ vanishing at $0$, and it is defined in terms of a Luxemburg type Gagliardo seminorm. When both Young function $G$ and its complementary function satisfy an appropriated growth behavior known as the $\Delta_2$ condition, in \cite{FBS} the following limit behavior of the seminorms was proved:
$$
\lim_{s\uparrow 1} (1-s)\iint_{\R^n\times\R^n} G\left(\frac{|u(x)-u(y)|}{|x-y|^s}\right)\frac{dxdy}{|x-y|^n} = \int_{\R^n} G_0 (|\nabla u|)\,dx,
$$
where, for $t\geq0$,  $G_0(t)=\int_0^t \int_{\mathbb{S}^{n-1}} G(r|\theta\cdot e|)\,d\mathcal{H}^{n-1}(\theta)\frac{dr}{r}$ is a Young function equivalent to $G$. The case of $G$ with a general growth behavior was covered in \cite{ACPS, ACPS1}. A further extension to Carnot groups can be found in \cite{CMSV}. The case of the magnetic fractional Orlicz-Sobolev spaces was studied in \cite{FBS1}.

Young functions include as typical examples power functions, i.e. $G(t)=t^p$, $p>1$, and logarithmic perturbation of powers such as $G(t)=t^p\log (1+t)$, $p>1$. Nevertheless, anisotropic and double-phase behaviors are not contemplated in this class, and functions of the type $t^{p(x)}$ for a suitable function $p(x)$, which are related with the fractional Sobolev spaces with variable exponent, are not covered by the previous results. See \cite{URD, BR} for an introduction and properties of these spaces. In this line of research, in \cite{K} is answered the question whether a Bourgain-Brezis-Mironescu (BBM) type result is true in the fractional Sobolev spaces with variable exponent $W^{s,p(x,y)}(\R^n)$, when $p\colon \R^n\times\R^n\to\R$ is such that $p(x,\cdot)$ is $\log-$H\"older continuous for any fixed $x\in\R$ and there are constants $p^\pm$ such that $1<p^-\leq p(x,y)\leq p^+<\infty$. The main result in \cite{K} establishes that for sufficiently smooth functions, let us say $u\in C^2_c(\R^n)$, it holds that
$$
\lim_{s\uparrow 1}s(1-s)\iint_{\R^n\times\R^n} \frac{|u(x)-u(y)|^{p(x,y)}}{|x-y|^{n+sp(x,y)}}\;dxdy = \int_{\R^n} K_{n,\overline{p}(x)} |\nabla u(x)|^{\overline{p}(x)}\,dx,
$$
where $\overline{p}(x):=p(x,x)$ and $K_{n,\overline{p}(x)}= \frac{1}{\overline{p}(x)}\int_{\mathbb{S}^{n-1}} |\theta\cdot e|^{\overline{p}(x)}\,d\mathcal{H}^{n-1}(\theta)$.

While the previous result holds for smooth functions, the author of \cite{K} proves that it does not hold for all functions in $W^{1,\overline{p}(x)}(\R^n)$, even when the variable exponent $p$ is smooth. This is in sharp contrast to the case when $p$ is constant. The reason for this is that the target space $W^{1,\overline{p}(x)}$ is too large for the previous BBM expression to be true in general.

The aim of this manuscript is to study the asymptotic behavior of energy functionals related to general anisotropic, non-local, non-standard growth spaces as $s\uparrow 1$. These spaces include fractional Orlicz-Sobolev spaces and fractional Sobolev spaces with variable exponents, which are examples of the aforementioned case.

Given a generalized Young function $G\colon \R^n\times\R^n \times [0,\infty) \to \R$ (see Section \ref{gral.young.functions} for details)  and $s\in (0,1)$ we consider the energy functional
$$
\J (u):=\iint_{\R^n\times\R^n} G(x,y,|D_s u(x,y)|)\,d\mu(x,y),
$$
where the $s-$H\"older quotient $D_s u$ and the measure $\mu$ are defined as
$$
D_su(x,y):=\frac{u(x)-u(y)}{|x-y|^s}, \qquad d\mu(x,y) := \frac{dxdy}{|x-y|^n}.
$$
Here, the functional $\J$ is well-defined when $u$ belongs to the fractional Musielak-Sobolev space $W^{s,G}(\R^n)$ considered in \cite{ABSS, ABSS2, BMO} (see Section \ref{spaces} for precise definitions).

In order to prove our results we assume some structural hypotheses on the generalized Young function $G$. First, we ask for a boundedness condition on $G$ with respect to $(x,y)\in\R^n\times\R^n$: there exist constants $0<C_1\leq C_2<\infty$ such that
\begin{equation} \tag{$H_1$} \label{Hi1'}
	C_1\leq\inf_{x,y\in\R^n}G(x,y,1)\leq G(x,y,1)\leq \sup_{x,y\in\R^n} G(x,y,1) \leq C_2.
\end{equation}
%It is not hard to see that $\J$ is Fr\'echet differentiable. In order to obtain an integral representation of the derivative $\mathcal{J}'(u)\in (W^{s,G}(\R^n))'$ we impose the symmetry condition for any $(x,y)\in\R^n\times\R^n$ and $t\geq 0$
%Secondly, we impose a symmetry condition for any $(x,y)\in \R^n\times\R^n$ and $t\geq 0$:
%\begin{equation} \tag{$H_2$} \label{Hi2'}
%	G(x,y,t)= G(y,x,t).
%\end{equation}
In order to analyze the behavior as $s\uparrow 1$, we impose the following  condition: 
\begin{equation} \tag{$H_2$} \label{Hi3'}
y \mapsto G(x,y,t)\;\; \mbox{is continuous}.
\end{equation}
Finally, we assume the following growth behavior for any $(x,y,t)\in \R^n\times\R^n\times[0,\infty)$: there exist constants $1<p^-\leq p^+<\infty$ such that
\begin{equation} \tag{$H_3$}  \label{Hi4}
p^-\leq \frac{tg(x,y,t)}{G(x,y,t)} \leq p^+.
\end{equation}
It is well known that $G$ and its complementary function satisfying $\Delta_2$ condition if and only if \eqref{Hi4} holds.

The dependence of the energy functional $\J$ on both $x$ and $y$ adds an extra level of difficulty when dealing with the limit behavior on the fractional parameter $s$. Given that our results include the case of fractional Sobolev spaces with variable exponents, we cannot expect to obtain a BBM-type formula beyond $C^2_c$ functions. Keeping these considerations in mind, our main result, stated in Theorem \ref{asym. behavior}, claims that for any $u\in C^2_c(\R^n)$,
$$
\lim_{s\uparrow 1} (1-s) \J(u) = \int_{\R^n} H_0(x,|\nabla u(x)|)\,dx,
$$
where the function $H_0$ is given by
$$
H_0(x,t)=\int_0^1 \int_{\mathbb{S}^{n-1}} G(x,x,t|w_n|r)\,d\mathcal{H}^{n-1}(w) \frac{dr}{r}
$$
and $w_n$ is the $n-$th coordinate of any point in $\mathbb{S}^{n-1}$.

As a consequence, we will obtain in Corollary \ref{Asym. behavior. norm} a BBM type inequality for norms. In Proposition \ref{prop.H0} it is proved that the limit function $H_0(x,t)$ is in fact equivalent to the generalized Young function $\overline{G}(x,t):=G(x,x,t)$.

\medskip 

Examples of functions fulfilling our hypotheses include:
\medskip

\begin{enumerate}[leftmargin=20pt]
\item[(i)] $G(x,y,t)=A(t)$
for any Young function $A$ (a convex function $A\colon [0,\infty)\to [0,\infty)$ with $A(0)=0$) such that $1<p^-\leq \frac{t A'(t)}{A(t)}\leq p^+<\infty$ for certain constants $p^+$ and $p^-$. In particular, for $A(t)=t^p$ and $A(t)=t^p\log (1+t)$, $1<p<\infty$;

\medskip

\item[(ii)] $G(x,y,t)=t^p+ a(x,y)t^q$ where $t\geq0$, $1<p<q<\infty$ and $a(x,y)$ is a non-negative bounded and continuous function in the second variable;

\medskip

\item[(iii)]$G(x,y,t)=t^{p(x,y)}$ with $t\geq0$, and $p(x,y)$ is a continuous function in the second variable such that $1<p^-\leq p(x,y)\leq p^+<\infty$ for all $(x,y)\in \R^n\times\R^n$.
\end{enumerate}
\medskip

These examples include equations defined  in fractional Orlicz-Sovolev spaces (see for instance \cite{ACPS,ACPS1,FBS,FBS1}), double phase problems (see for instance \cite{BCM 1,BCM 2}) and Sobolev spaces of variable exponent (see for instance \cite{BR,K,URD}).

The limit formula we obtain in Theorem \ref{asym. behavior} is not valid for the entire space $W^{1,\overline G}(\R^n)$ for an arbitrary generalized Young function $G$ that satisfies $\eqref{Hi1'}$--$\eqref{Hi4}$, as demonstrated by some counterexamples in \cite{K}. However, in Section \ref{examples}, we prove its validity for certain particular classes. Additionally, for any generalized Young function, the limit formula holds in a smaller Sobolev space, which is determined by the bounds $p^\pm$ given in \eqref{Hi4}, as shown in Corollary \ref{corolary}.

It is worth noting that for the fractional Sobolev space with variable exponent (i.e., example (iii) above), our result imposes weaker assumptions on $p(x,y)$ than those required in \cite{K}.

Recently, there has been consideration of fractional anisotropic spaces where the functions have different fractional regularity and integrability in each coordinate direction, as seen in \cite{CKW,DB}. The techniques used in our results enable us to study energy functionals where the $s$-H\"{o}lder quotient depends solely on a direction, that is, 
$$
D_s^k u(x,h):=\frac{u(x-he_k)-u(x)}{|h|^s}, \quad \text{ with } k\in\{1,\ldots, n\},
$$
being $e_k$ the $k-$th canonical vector in $\R^n$. More precisely, in Theorem \ref{thm-anisotropic} we prove that, for $u\in C^2_c(\R^n)$, 
$$
\lim_{s\uparrow 1}(1-s) \int_{\R^n} \int_{\R} G(x,x-he_k,|D_s^k u(x,y)|)\,\frac{dhdx}{|h|} = \int_{\R^n} H_0\left(x,\left|\frac{\partial u(x)}{\partial x_k} \right|\right)\,dx,
$$
where in this case the limit function $H_0$ is defined as
$$
H_0(x,t)=2 \int_0^1 G(x,x,tr) \frac{dr}{r}.
$$

We leave some open questions related to the asymptotic behavior of these energies. Due to the high dependence on spatial coordinates, study whether a BBM type formula holds for sequences of functions (depending on $s$) is a challenging task. This question remains unanswered even in the case of the fractional Laplacian with variable exponent. Due to the lack of this result, we were not able to obtain a BBM type formula for norms (see Corollary \ref{Asym. behavior. norm}), although we suspect this is valid. Another interesting yet highly nontrivial point is understanding the behavior of the energies as the fractional parameter $s$ approaches 0, in the spirit of the seminal work of Maz'ya and Shaposhnikova \cite{MZ}.

The manuscript is organized as follows: in Section \ref{sec2} we introduce the notion of generalized Young functions and the definition of fractional Musielak-Sobolev spaces; Section \ref{sec3} contains the proof of our main result; in Section \ref{examples} we introduce some application examples considered in this theory; finally in section \ref{sec5} we give some extensions and final remarks.

%\textcolor{red}{Comment: if we consider  $G(x,y,t)=t^{p(x,y)}$ with $p(x,y)$ continuous, symmetric and $1<p^-\leq p(x,y)\leq p^+<\infty$ for all $(x,y)\in \R^n\times\R^n$, then  hypotheses $\eqref{Hi1'}$--$\eqref{Hi4}$ on $G$ are fullfilled and then the BBM formula holds for $C^2_c(\R^n)$ functions.
%\\
%The question is: why Kim need then the log-H\"older continuity of $p(x,y)$ in the variable $y$ (we are using only the continuity). I don't get that part.
%}\\
%
%\textcolor{blue}{Answer: In \cite[Page 5, Eq. (3.7)]{K}, the author has to deal with an integral involving three terms. In order to deal with the terms 
% \[
%  \vert \nabla u(x)\vert^{p(x,x+\rho A\xi)} \quad \mbox{and} \quad \vert \xi_{n}\vert^{p(x,x+\rho A\xi)},
% \] 
%the author used the continuity of the power (see estimates (3.3) and (3.4)). However, the same argument can not be applied on the term
% \[
%  \rho^{-1+(1-s)p(x,x+\rho A\xi)}.
% \] 
%For overcome this difficulty, the author used the ``log-H\"{o}lder continuity'' (see (1.4) and (3.8)). We improve this hypothesis by estimating in a ball and outside the ball, using the fact that the Young function is Lipschitz (see \eqref{Lipschitz}) and applying Lebesgue Convergence Dominated Theorem jointly with the technical Lemma \ref{conv. Young function}.}

\section{Preliminaries} \label{sec2}

%\textcolor{red}{In this section we collect preliminary concepts of the theory of generalized Young functions and fractional order Musielak-Sobolev spaces, which will be used throughout the paper.}
\subsection{Generalized Young functions}  \label{gral.young.functions}
We consider the class of functions  $G\colon \R^n\times\R^n\times [0,\infty)\to [0,\infty)$ given by
$$
G(x,y,t)=\int_0^t g(x,y,s)\,ds,
$$
where $g\colon \R^n\times\R^n\times[0,\infty)\to [0,\infty)$ is a  function satisfying
\begin{itemize}[leftmargin=20pt]
\item[(i)] $g(x,y,0)=0$ and $g(x,y,t)>0$ for any $(x,y)\in\R^n\times\R^n, t>0$;
\item[(ii)] $g(\cdot,\cdot,t)$ is nondecreasing in $[0,\infty)$;
\item[(iii)] $g(\cdot,\cdot,t)$ is right continuous in $[0,\infty)$ and $\lim_{t\to\infty} g(\cdot,\cdot,t)=\infty$.
\end{itemize}

It is not hard to show that hypotheses $(i)$--$(iii)$ imply that $G(x,y,t)$ is continuous, locally Lipschitz continuous, strictly increasing and convex in $t\geq 0$ for any $(x,y)\in \R^n\times\R^n$. Moreover, $G(x,y,0)=0$ and $G(x,y,t)$ has a sublinear decay at zero and a superlinear growth
near infinity in the variable $t$:
$$
\lim_{t\to 0^+}\frac{G(x,y,t)}{t}=0, \qquad \lim_{t\to\infty}\frac{G(x,y,t)}{t}=\infty.
$$
A function $G$ defined as before is known as \emph{generalized Young function}.

Throughout this paper we will assume generalized Young functions satisfying the structural assumptions \eqref{Hi1'}--\eqref{Hi4}.
 
We remark that, from \eqref{Hi1'} and \eqref{Hi4} it follows that (see \cite[Lemma 2.2]{ABSS}).
\begin{lema}\label{prop. 1}
For any $(x,y)\in\R^n\times\R^n$ and $a,b\geq0$ it holds that
$$
\min\{a^{p^-}, a^{p^+}\}G(x,y,b)\leq G(x,y,ab)\leq \max\{a^{p^-}, a^{p^+}\}G(x,y,b),
$$
$$
 C_1\min\{b^{p^-}, b^{p^+}\}\leq G(x,y,b)\leq C_2\max\{b^{p^-}, b^{p^+}\}.
$$	
\end{lema}

\medskip

We define the \emph{complementary Young function} of $G$ as
$$
\tilde G(x,y,t):=\sup_{w\geq 0} (tw-G(x,y,w)).
$$
Consequently, the following Young's inequality holds for $a,b\geq 0$ and $(x,y)\in\R^n\times\R^n$:
$$
ab\leq G(x,y,a)+\tilde G(x,y,b).
$$
The complementaty Young function satisfies the following relation.
\begin{lema} \label{lema1}
Let $G$ be a Young function and $\tilde G$ its complementary function. Then,
$$
\tilde G(x,y,g(x,y,t)) \leq p^+G(x,y,t),
$$
for any $(x,y,t)\in \R^n\times \R^n\times [0,+\infty)$.
\end{lema}
\begin{proof}
Let $t\geq 0$ be fixed and for each $w \in [0, +\infty)$ consider the function 
$$f(x,y,w):=g(x,y,t)w-G(x,y,w).$$ 
Observe that $\displaystyle\frac{d}{dw}f(x,y,w)= 0$ when $w=t$ and 
$$
\begin{aligned}
	\frac{d}{dw}f(x,y,w)=g(x,y,t)-g(x,y,w) \geq 0, &\quad \mbox{when } \, w< t,\\
	\frac{d}{dw}f(x,y,w)=g(x,y,t)-g(x,y,w) \leq 0, &\quad \mbox{when } \, w>t,
\end{aligned}
$$
since $g(\cdot,\cdot,t)$ is nondecreasing  in $t$. Therefore, $f(x,y,t) \geq f(x,y,w)$ for any $w \in [0, +\infty)$, and this means that
$$
\tilde G(x,y,g(x,y,t))=\sup_{w\geq 0} f(x,y,w) = g(x,y,t)t-G(x,y,t).
$$
Due to \eqref{Hi4}, this expression can be bounded by $p^+ G(x,y,t)$.
\end{proof}

%The Jensen's inequality in this settings reads as:
%\begin{prop}[reference?]
%$$
%G\left(x,h,\frac{1}{|\Omega|}\int_\Omega u(y)\,dy \right) \leq \frac{1}{|\Omega|} \int_\Omega G(x,h,u(y)) \,dy.
%$$
%\end{prop}

\subsection{Musielak-Sobolev spaces} \label{spaces}
Given an open set $\Omega\subset \R^n$ and a generalized Young function $G\colon \Omega\times\Omega\times\R \to [0,+\infty)$  satisfying \eqref{Hi4}, we consider the generalized Young function $\overline{G}:\Omega\times [0,+\infty)\to[0,+\infty)$ given by  
$$
\overline{G}(x,t):=G(x,x,t)=\int_0^t \overline{g}(x,s)\,ds,
$$
where $\overline{g}(x,t)=g(x,x,t)$ for any $(x,t)\in \Omega \times [0,+\infty)$. It is easy to verify that \eqref{Hi4} implies 
	$$
p^{-}\leq \frac{t \overline{g}(x,t)}{\overline{G}(x,t)}\leq p^{+}, \quad \mbox{for any} \hspace{0,2cm} x\in \Omega \hspace{0,2cm}  \mbox{and} \hspace{0,2cm} t>0.
$$

Given $s\in(0,1)$ we consider the spaces
\begin{align*}
&L^{\overline{G}(\cdot)}(\Omega):= \left\{u\colon \Omega\to \R \text{ measurable}: \JJ(u) <\infty\right\},\\
&W^{s,G(\cdot,\cdot)}(\Omega):=\left\{u\in L^{\overline{G}(\cdot)}(\Omega)\colon \J(u)<\infty\right\}
\end{align*}
where
\begin{align*}
\J (u):=\iint_{\R^n\times\R^n} G(x,y,|D_s u(x,y)|)\,d\mu(x,y), \qquad \JJ(u):=\int_\Omega \overline{G}(x,|u(x)|)\,dx
\end{align*}
and the $s-$H\"older quotient $D_s u$ and the measure $\mu$ are defined as
$$
D_su(x,y):=\frac{u(x)-u(y)}{|x-y|^s}, \qquad d\mu(x,y) := \frac{dxdy}{|x-y|^n}.
$$

These spaces are separable and reflexive Banach spaces endowed with the Luxemburg norms
$$
\|u\|_{\overline G(\cdot)}:=\inf\left\{ \lambda>0\colon \int_\Omega \overline{G}\left(\frac{u}{\lambda}\right)\,dx\leq 1\right\}, \qquad 
\|u\|_{s,G(\cdot,\cdot)}:=\|u\|_{\overline{G}(\cdot)} + [u]_{s,G(\cdot,\cdot)},
$$
respectively, with corresponding seminorm
$$
[u]_{s,G(\cdot,\cdot)}:=\inf\left\{ \lambda>0\colon \J \left(\frac{u}{\lambda}\right)\,dx\leq 1\right\}.
$$

We also consider the local space
$$
W^{1,\overline{G}(\cdot)}(\Omega):=\left\{ u \in L^{\overline G(\cdot)}\colon |\nabla u| \in L^{\overline G(\cdot)} \right\}
$$
with the corresponding norm  given by $\|u\|_{1,\overline{G}(\cdot)}:= \|u\|_{\overline G(\cdot)} + \|\nabla u\|_{\overline G(\cdot)}$.

\subsection{Notation}
In order to simplify the reading of the paper, we drop the dependence on $x$ and $y$ in the spaces and we just write
$L^{\overline G}(\Omega)$, $W^{1,\overline G}(\Omega)$ and $W^{s,G}(\Omega)$ instead of $L^{\overline G(\cdot)}(\Omega)$, $W^{1,\overline{G}(\cdot)}(\Omega)$ and $W^{s,G(\cdot,\cdot)}(\Omega)$, respectively. Similarly, we write  $\|u\|_{\overline G}$, $\|u\|_{1,\overline G}$, $\|u\|_{s,G}$ and $[u]_{s,G}$ in place of  $\|u\|_{\overline G(\cdot)}$, $\|u\|_{1,\overline G(\cdot)}$, $\|u\|_{s,G(\cdot,\cdot)}$ and $[u]_{s,G(\cdot,\cdot)}$, respectively. 

\medskip

Throughout the paper, the following notation will also be
used:
\begin{enumerate}[leftmargin=20pt] \itemsep2pt \parskip3pt \parsep2pt
\item[$\circ$] The norm in the usual space $L^p(\R^n)$, $p\in [1,\infty]$, is denoted by $\|u\|_p$.
\item[$\circ$] The characteristic function of a subset $B \subset \R^n$ is denoted by $\chi_{B}$. 
	
	\item[$\circ$] The $(n-1)$-dimensional unit sphere in $\R^n$ is denoted by $\mathbb{S}^{n-1}$.
	
	\item[$\circ$] The volume of the unit ball in $\R^n$ is denoted by $\omega_n$.
	
	\item[$\circ$] The volume of the $(n-1)$-dimensional unit sphere in $\R^n$ is then $n\omega_n$.
\end{enumerate}

\section{The BBM formula} \label{sec3}
 
We begin this section by proving a technical lemma. These result plays a very important role in the proof of the BBM formula.

\begin{lema}\label{conv. Young function}
For any $x\in \mathbb{R}^n$ and $t\geq0$, it holds that
$$
\lim_{s\uparrow 1} (1-s) \int_{0}^{1} \int_{\mathbb{S}^{n-1} }G(x,x-rw, t|w_n|r^{1-s}) \; \,dS_w \frac{dr}{r}= H_0(x,t),
$$
where 
\begin{equation}\label{limit function}
	H_0(x,t):=\int_0^1 \int_{\mathbb{S}^{n-1}} G(x,x,t|w_n|r)\,dS_w \frac{dr}{r}
\end{equation}
and $w_n$ is the $n-$th coordinate of any point in $\mathbb{S}^{n-1}$.
\end{lema}
\begin{proof}
By performing the change of variables $\rho=r^{1-s}$, we deduce that
\begin{align*}
\int_{0}^{1} \int_{\mathbb{S}^{n-1} }G(x,x-rw,& t|w_n|r^{1-s}) \; \,dS_w \frac{dr}{r}\\
&=\frac{1}{1-s} \int_0^1 \int_{\mathbb{S}^{n-1}} G\left(x,x-\rho ^{\frac{1}{1-s}}w,t|w_n| \rho \right) \,dS_w\frac{d\rho}{\rho}.
\end{align*}
Since $0<\rho<1$ and $G(x,\cdot, t)$ is continuous at $x$ by \eqref{Hi3'}, it follows that 
$$\lim_{s\uparrow 1}G\left(x,x-\rho ^{\frac{1}{1-s}}w, t|w_n| \rho \right)=G\left(x,x, t|w_n| \rho \right).$$
Moreover, using \eqref{Hi1'} and Lemma \ref{prop. 1},  we have
%as $u\in C_c^2(\mathbb{R}^N)$
\begin{align*}
G\left(x, x-\rho^{\frac{1}{1-s}}w ,t|w_n| \rho \right)\rho^{-1}&\leq \sup_{x,y\in \mathbb{R}^N}G(x,y,1)|w_n|^{p^-}\rho^{p^- -1}\max\{t^{p^-},t^{p^+}\}\\
& 
\leq C_2 \max\{t^{p^-},t^{p^+}\} \in L^1((0,1) \times \mathbb{S}^{n-1}).
\end{align*}
Therefore, the result follows by Lebesgue's Dominated Convergence Theorem.
\end{proof}

The next proposition ensures that $H_0(x,t)$ is a generalized Young function and that it is equiva\-lent to $\overline{G}(x,t)$. Thereby, the Musielak-Orlicz spaces $L^{\overline{G}}(\mathbb{R}^n)$ and $L^{H_0}(\mathbb{R}^n)$ are the same.

\begin{prop} \label{prop.H0} 
	The function $H_0$ defined in \eqref{limit function} is a generalized Young function. Furthermore, there exist positive constants $C_1$ and $C_2$ such that
	\begin{equation}\label{equiv1}
		C_1  \overline{G}(x,t)\leq  H_0(x,t) \leq C_2 \overline{G}(x,t),
	\end{equation}
	for any $(x,t)\in \R^n \times [0,+\infty)$.
\end{prop}

\begin{proof}
	
We prove first that $H_0$ is a generalized Young function. Note that, making use of the change of variables $\rho=tr$, we can write
$$H_0(x,t)=\int_0^t \int_{\mathbb{S}^{n-1}} G(x,x,|w_n|\rho)\,dS_w \frac{d\rho}{\rho}= \int_0^t h_0(x,\rho)\; d\rho, $$
where
$$h_0(x,\rho):=
\begin{cases}
	\displaystyle\int_{\mathbb{S}^{n-1}} \frac{G(x,x,|w_n|\rho)}{\rho} \,dS_w,& \quad \mbox{when }\, \rho\in (0,+\infty),\\ 
	0, & \quad \mbox{when }\, \rho=0.
\end{cases}
$$

The hypotheses $(i)$ and $(ii)$ are an immediate consequence of increasing monotonicity of function $\frac{G(x,x,|w_n|t)}{t}$ in $t\in(0,+\infty)$ for any $x\in \R^n$. Finally, since $G(x, x, t)$ is continuous and has superlinear growth near infinity in the variable $t$ for any $x\in\R^n$, we conclude that $h_0$ satisfies $(iii)$. 
	
It remains to prove the equivalence \eqref{equiv1}. Given $x\in\mathbb{R}^n$ and $t\geq0$, it follows from the monotonicity of  $G(x,y,t)$ at $t$ and Lemma \ref{prop. 1} that 
\begin{align*}
	H_0(x,t) &\leq \int_0^1 \int_{\mathbb{S}^{n-1}} G(x,x,tr)\,dS_w \frac{dr}{r} \\
	&
	\leq  \int_0^1 \int_{\mathbb{S}^{n-1}} G(x,x,t) r^{p^--1}\,dS_w dr\\
	& =\frac{n\omega_n}{p^-}G(x,x,t).
\end{align*}
On the other hand, by using Lemma \ref{prop. 1}, we have
\begin{align*}
	H_0(x,t) &\geq  \int_0^1 \int_{\mathbb{S}^{n-1} } (|w_n|r)^{p^+}G(x,x,t) dS_w \frac{dr}{r}\\
	& = G(x,x,t) \left( \int_{\mathbb{S}^{n-1} } |w_n|^{p^+}dS_w\right) \left(\int_0^1 r^{p^{+}-1}\, dr\right)\\ 
	& =C(n,p^+)G(x,x,t).
\end{align*}
This ends the proof.
\end{proof}

In the following we establish our main result.

\begin{thm}[BBM formula]\label{asym. behavior}
Let $u\in C^2_c(\R^n)$. Then, 
\begin{align}\label{eq.asym.behavior}
	\lim_{s\uparrow 1} (1-s) \J(u) = \int_{\R^n} H_0(x,|\nabla u(x)|)\,dx,
\end{align}
where $H_0$ was defined in \eqref{limit function}.
%$$
%H_0(x,t)=\int_0^1 \int_{\mathbb{S}^{n-1}} G(x,x,t|w_n|r)\,dS_w \frac{dr}{r}.
%$$
\end{thm}

\begin{proof}
For a fixed $x\in \R^n$, we split the integral
\begin{align*} 
\int_{\R^n} G(x,x-h,|D_su(x,x-h)|)\frac{dh}{|h|^n}&= 
 \int_{|h|<1} G(x,x-h,|D_su(x,x-h)|)\frac{dh}{|h|^n}\\ 
 &\quad+ \int_{|h|\geq 1} G(x,x-h,|D_su(x,x-h)|)\frac{dh}{|h|^n}\\
 &=\colon I_1+I_2.
\end{align*}
Let us deal with $I_2$. By using the monotonicity and  convexity of $G(x,y, t)$ in $t$, \eqref{Hi1'} and Lemma \ref{prop. 1}, we have that
\begin{align} \label{bound.I2}
\begin{split}
I_2 &\leq
\int_{|h|\geq 1} G\left(x,x-h, \frac{2\|u\|_\infty}{|h|^s} \right)\frac{dh}{|h|^n}\\
&\leq
\int_{|h|\geq 1} G\left(x,x-h,  2\|u\|_\infty  \right)\frac{dh}{|h|^{n+s}}\\
&\leq C_2 2^{p^+}\max\{\|u\|_\infty^{p^-},\|u\|_\infty^{p^+}\} \frac{n\omega_n}{s},
\end{split}
\end{align}
from where we obtain
\begin{equation} \label{int2}
\lim_{s\uparrow 1} (1-s)I_2 = 0.
\end{equation}
Let us now estimate $I_1$. Since $u$ is smooth and $G(x,y,t)$ is Lipschitz continuous in $t$,
\begin{align*}
\Big| G\left(x,x-h,|D_s u(x,x-h)| \right) - G&\left(x,x-h,\left| \nabla u(x)\cdot \frac{h}{|h|^s}\right|  \right)\Big|\nonumber \\
&\leq  L \frac{|u(x)-u(x-h)-\nabla u(x)\cdot h|}{|h|^s}\leq C |h|^{2-s},
\end{align*}
where $L$ is the Lipschitz constant of $G$ on the interval $[0,\|\nabla u\|_\infty]$ and $C$ depends on the $C^2-$norm of $u$. Since 
$$
\int_{|h|<1} |h|^{2-s-n}\,dh = \frac{n\omega_n}{2-s},
$$
it follows that
\begin{align*}
\lim_{s\uparrow 1} (1-s)I_1 =\lim_{s\uparrow 1} (1-s)\int_{|h|<1} G\left(x,x-h,\left| \nabla u(x)\cdot \frac{h}{|h|^s} \right| \right) \frac{dh}{|h|^n}.
\end{align*}
Observe that, by using spherical coordinates, we can write 
\begin{align*}
\int_{|h|<1} G\left(x,x-h,\left|\nabla u(x)\cdot \frac{h}{|h|^s} \right| \right)& \frac{dh}{|h|^n} =
\int_0^1 \int_{|h|=r} G\left(x,x-h , \left| \nabla u(x)\cdot \frac{h}{|r|^s} \right|  \right) \,dS_h\frac{dr}{r^n}\\
&=
\int_0^1 \int_{\mathbb{S}^{n-1}} r^{n-1}G\left(x,x-rw,|\nabla u(x)\cdot w| r^{1-s} \right) \,dS_w\frac{dr}{r^n} \\
&=
\int_0^1 \int_{\mathbb{S}^{n-1}} G\left(x,x-rw,|\nabla u(x)||w_n| r^{1-s} \right) \,dS_w\frac{dr}{r},
\end{align*}
where in the last equality we have performed a rotation such that $\nabla u(x)=|\nabla u(x)|e_n$. Thus, in view of Lemma \ref{conv. Young function}, we have that
\begin{equation} \label{int1}
\lim_{s\uparrow 1}(1-s)I_1=H_0(x,|\nabla u(x)|),
\end{equation}
for any $x\in\mathbb{R}^n$. 

Gathering \eqref{int2} and \eqref{int1}, we conclude that, for any $x\in\R^n$,
$$
\lim_{s\uparrow 1}(1-s)\int_{\R^n} G(x,x-h,|D_su(x,x-h)|)\frac{dh}{|h|^n} = H_0(x,|\nabla u(x)|).
$$

\medskip

%\textcolor{red}{comment: I changed a bit the part of the majorant since in principle we don't know if $|u(x)|^{p^+}+|u(x)|^{p^-} \in L^1(\R^n)$, right?}

In order to complete the proof, it only remains to show the existence of an integrable majorant for $(1-s)F_s$, where
$$
F_s(x):=\int_{\R^n}G(x,x-h,|D_su(x,x-h)|)\frac{dh}{|h|^n}.
$$
Since $u\in C_c^2(\R^n)$, we can assume without loss of generality that $\text{supp}\,(u)\subset B_R(0)$ with $R>1$. First we analyze the behavior of $F_s(x)$ for small values of $x$. When $|x|<2R$ we can write
$$
|F_s(x)| = \left(\int_{|h|<1} + \int_{|h|\geq 1}\right)   G(x,x-h,|D_su(x,x-h)|)\frac{dh}{|h|^n} := I_1 + I_2.
$$
By using the expression of $I_1$ obtained before, together with the convexity and monotonicity of $G$, and Lemma \ref{prop. 1}, we obtain
\begin{align} \label{bound.F1}
	\begin{split}
		I_1&\leq\int_{|h|<1} G\left(x,x-h,\left|\nabla u(x)\cdot \frac{h}{|h|^s} \right| \right) \frac{dh}{|h|^n} +C\\
		&\leq \int_{|h|<1} |h|^{1-s-n} G(x,x-h, \|\nabla u\|_\infty)\,dh+ C\\
		&\leq C_2\frac{n\omega_n}{1-s}(\|\nabla u\|_\infty^{p^+} + \|\nabla u\|_\infty^{p^-})+ C.
	\end{split}
\end{align}
Furthermore, it is follows from \eqref{bound.I2} that
\begin{align} \label{bound.F2}
I_2 \leq 2^{p^+}\frac{C_2 n\omega_n}{s} (\|u\|^{p+}_\infty+\|u\|^{p-}_\infty).
\end{align}
On the other hand, when $|x|\geq 2R$, $u$ vanishes and we have that
$$
F_s(x)=\int_{B_R(0)}G\left(x,y,\frac{|u(y)|}{|x-y|^s}\right)\frac{dy}{|x-y|^n}.
$$
Since $|x-y|\geq |x|-|y|\geq \frac12|x|$, from the monotonicity of $G$ and  Lemma \ref{prop. 1}, we get
\begin{align} \label{bound.F3}
\begin{split}
|F_s(x)|&\leq 
\frac{2^n}{|x|^n}\int_{B_R(0)}G\left(x,y,\frac{2^s|u(y)|}{|x|^s}\right)\,dy\\
&\leq 
\frac{2^{n+sp^+}}{|x|^{n+s}}\int_{B_R(0)}G\left(x,y,|u(y)|\right)\,dy\\
&\leq 
\frac{2^{n+p^+}}{|x|^{n+\frac12}} C_2 n\omega_n R^n (\|u\|_\infty^{p^+} + \|u\|_\infty^{p^-}),
\end{split}
\end{align}
for any $s\geq \frac12$. Finally, from \eqref{bound.F1}, \eqref{bound.F2} and \eqref{bound.F3} we obtain that
$$
(1-s)|F_s(x)|\leq C\left( \chi_{B_{2R}(0)} + |x|^{-n-\frac12} \chi_{B_{2R}(0)^c}\right) \in L^1(\R^n),
$$
where $C>0$ is a constant depending of $n$, $p^\pm$ and $u$, but independent of $s$. 

Therefore, the result follows from the Lebesgue's Dominated Convergence Theorem  for any $u\in C^2_c(\R^n)$.
\end{proof}
%{\color{red}
%\begin{rem}
%	the Bourgain–Brezis–Mironescu type result is not true in general for fractional  Musielak-Sobolev spaces under the assumptions \eqref{Hi1'}-\eqref{Hi4}. Precisely, the result is true only for functions that are sufficiently regular. For example, M. Kim \cite{M. Kim} proves that the result can fail for fractional Sobolev spaces with variable exponents, even when the exponent is smooth. More specifically, the author shows that there exists a smooth variable exponent $p(\cdot,\cdot)$ such that the Proposition \ref{asym. behavior} does not hold for some function in Sobolev space with variable exponent associated. 
%\end{rem}
%}

Observe that Theorem \ref{asym. behavior} holds for any smooth function, but in general could be false in the space $W^{1,\overline G}(\R^n)$ as showed in \cite[Corollary 1.3]{K}.  Although that space is too large for \eqref{eq.asym.behavior} to be true, as a direct consequence of Lemma \ref{prop. 1} and the Theorem \ref{asym. behavior} in the usual fractional Sobolev spaces, we get the following.

\begin{cor} \label{corolary}
	The equality \eqref{eq.asym.behavior} holds for any $u\in W^{1,p^-}(\R^n)\cap W^{1,p^+}(\R^n)$.
\end{cor}

\begin{proof}
For any $u\in W^{1,p^-}(\R^n)\cap W^{1,p^+}(\R^n)$, we take a sequence $\{u_k\}_{k\in\mathbb{N}}\subset C^2_c(\R^n)$ such that $u_k \to u$ in $W^{1,p^-}(\R^n)$ and $	W^{1,p^+}(\R^n)$.  Without loss of generality, we may assume that $u_k\to u$ a.e. in $\mathbb{R}^n$. Observe that 
\begin{align*}
	|(1-s)\J(u) - \mathcal{J}_{1,H_0}(|\nabla u|)|&\leq (1-s)|\J(u) - \J(u_k)|\\
	& \quad + |(1-s) \J(u_k) - \mathcal{J}_{1,H_0}(|\nabla u_k|)|\\ 
	& \quad + |\mathcal{J}_{1,H_0}(|\nabla u_k|) - \mathcal{J}_{1,H_0}(|\nabla u|)|\\
	& =:I_1 + I_2 + I_3,
\end{align*}
where $\mathcal{J}_{1,H_0}(u):=\int_{\R^n}H_0(x,|u(x)|)\,dx$. 
Using Lemma \ref{prop. 1}, we can verify that $W^{1,p^-}(\R^n)\cap W^{1,p^+}(\R^n) \subset W^{1,\overline{G}}(\R^n)$ and $u_k\to u $ in $ W^{1,\overline{G}}(\R^n)$. This fact together with the Proposition \ref{prop.H0} imply that $I_3 \to 0$ as $k\to \infty$.
	
By Lemma \ref{prop. 1} and \cite[Theorem 1]{BBM}, we have that
\begin{align*}
	\J(v)&\leq C_2\iint_{\R^n\times\R^n} (|D_sv(x,y)|^{p^-} + |D_sv(x,y)|^{p^+} )\;d\mu\\
	%&=C_2 \iint_{\R^n\times\R^n} \left( \frac{|v(x)-v(y)|^{p^-}}{|x-y|^{n+sp^-}} + \frac{|v(x)-v(y)|^{p^+}}{|x-y|^{n+sp^+}} \right)\,dxdy\\
	& \leq \frac{C_2n \omega_n}{p^-}\left[\frac{1}{1-s} \left( \|\nabla v\|_{p^-}^{p^-} + \| \nabla v\|_{p^+}^{p^+} \right) + \frac{2^{p^+}}{s} \left( \|v\|_{p^-}^{p^-} + \|v\|_{p^+}^{p^+} \right)\right],
\end{align*}
for any $v \in W^{1,p^-}(\R^n)\cap W^{1,p^+}(\R^n)$. From where we deduce that $\J(u_k - u) \to 0$ as $k\to\infty$. Then, by \cite[Proposition 3.7]{AACS}, $I_1 \to 0$ as $k\to\infty$. Thus, for any $\varepsilon>0$, we can take $k$ enough large such that
$$
|(1-s)\J(u) - \mathcal{J}_{1,H_0}(|\nabla u|)|\leq \varepsilon +I_2.
$$
Therefore, taking the limit as $s \uparrow1$ and invoking Theorem \ref{asym. behavior}, the result follows.
\end{proof}

The asymptotic behavior of modulars stated in Theorem \ref{asym. behavior}  gives indeed a BBM type inequality formula for norms. For this purpose, instead of the seminorm $[\cdot]_{s,G}$ defined in Section \ref{sec2} as
$$
[u]_{s,G}:=\inf\left\{ \lambda>0\colon \J \left(\frac{u}{\lambda}\right)\,dx\leq 1\right\},
$$
we consider the equivalent one defined as
$$
[[u]]_{s,G}:=\inf\left\{ \lambda>0\colon (1-s)\J \left(\frac{u}{\lambda}\right)\,dx\leq 1\right\}.
$$
In this case, the definition of the seminorm gives that
	\begin{equation} \label{eqqq1}
		(1-s)\J\left(\frac{u}{[[u]]_{s,G}}\right)\leq 1.
	\end{equation}
The following result establishes a BBM type inequality formula for norms.
\begin{cor}\label{Asym. behavior. norm}
	Let $u\in C_c^2(\R^n)$. Then
	$$
\limsup_{s\uparrow 1}[[u]]_{s,G} \leq \|\nabla u\|_{H_0},
	$$
	where $H_0$ was defined in \eqref{limit function}.
\end{cor}

\begin{proof}
Since we are assuming the $\Delta_2$ condition, one can check that bounded modulars are equiva\-lent to bounded seminorms.  Indeed, if $\int_{\R^n} H_0(x,|\nabla u|)\,dx \leq C$ for some $C\geq 1$, then by Lemma \ref{prop. 1}, we have
$$
\int_{\R^n} H_0\left(x, \frac{|\nabla u|}{C^\frac{1}{p^-}} \right)\,dx \leq  \frac{1}{C} \int_{\R^n} H_0(x,|\nabla u|)\,dx\leq 1,	$$
which gives $\|\nabla u\|_{H_0} \leq C^\frac{1}{p^-}$. Otherwise, if $\|\nabla u\|_{H_0} \leq C$, with $C\geq 1$,
$$
1\geq \int_{\R^n} H_0\left(x, \frac{|\nabla u|}{\|\nabla u\|_{H_0}}\right)\,dx \geq 
\int_{\R^n} H_0\left(x, \frac{|\nabla u|}{C}\right)\,dx  \geq \frac{1}{C^{p^+}} \int_{\R^n} H_0(x, |\nabla u|)\,dx
$$
and therefore, $\int_{\R^n} H_0(x,|\nabla u|)\,dx \leq C^{p^+}$. The same argument can be applied for $[[\cdot]]_{s,G}$ and $(1-s)\J(\cdot)$.

Let $u\in C_c^2(\R^n)$. On the one hand, using Theorem \ref{asym. behavior}, the definition of the norm $\|\cdot\|_{H_0}$ leads to
\begin{align*}
	\lim_{s\uparrow 1} (1-s)\J \left(\frac{u}{  \|\nabla u\|_{H_0}    } \right)  = \int_{\R^n} H_0\left(x, \frac{|\nabla u|}{ \|\nabla u\|_{H_0}} \right)\, dx = \colon\ell \leq 1.
\end{align*}
Thus, by definition of limit, there exists $\ve_s>0$ such that $\ve_s \to 0$ as $s\to 1^+$ and
$$
\left|(1-s)\J \left(\frac{u}{  \|\nabla u\|_{H_0}    } \right) - \ell \right| \leq \ve_s.
$$
In particular, this gives that
$$	
(1-s)\J \left(\frac{u}{  \|\nabla u\|_{H_0}    } \right) \frac{1}{1 +\ve_s}  \leq 1.
$$
Observe that $\frac{1}{1+\ve_s}<1$, so, by the convexity of $G$, we have
\begin{align*}
	1\geq (1-s)\J \left(\frac{u}{  \|\nabla u\|_{H_0}    } \right) \frac{1}{1 +\ve_s} \geq (1-s)\J \left(\frac{u}{ (1+\ve_s) \|\nabla u\|_{H_0}    } \right) .
\end{align*}
Then, by definition of the norm, we obtain that
$$
[[u]]_{s,G}\leq (1+\ve_s) \|\nabla u\|_{H_0},
$$
from where, taking the limit in $s$, we conclude that
$$
\limsup_{s\uparrow 1}[[u]]_{s,G} \leq \|\nabla u\|_{H_0}.
$$
This concludes the proof.
\end{proof}

%\color{blue}
%\begin{rem}
%	As mentioned in the introduction, a BBM type formula for sequences of functions seems to be a fundamental key to prove the other inequality for norms. We conjecture that the equality holds.
%\end{rem}
%\normalcolor

\section{Some examples}\label{examples}
Even though for an arbitrary generalized Young function $G$, the Theorem \ref{asym. behavior} may not be extended beyond $C^2_c(\R^n)$ functions, in this section we illustrate  some examples where Theorem \ref{asym. behavior} holds in a suitable space.

\subsection{Convex functions}
The Theorem \ref{asym. behavior} holds in particular when the Young function does not depend on the spatial variables  i.e., for any convex function $A\colon [0,\infty)\to[0,\infty)$ with $A(0)=0$ by writing  $G(x,y,t)=A(t)$. In this case it is recovered the results from  \cite{ACPS} and \cite{FBS}. In particular, this includes the case of powers given  in \cite{BBM}.

\subsection{Double phase functions}\label{exam1}

Let $1<q\leq p<\infty$ and consider a function $a(x,y)$ continuous in the second variable such that for some constants $a_\pm$
$$
0<a_-\leq a(x,y) \leq a_+<\infty, \quad \mbox{for any } (x,y)\in \R^n\times\R^n.
$$
Under these assumptions we consider the generalized Young function $G(x,y,t)=t^q + a(x,y)t^p$, $t\geq 0$. In this case, 
$$
\J(u)\colon=\mathcal{J}_{s,p,q}(u)=\iint_{\R^n\times\R^n} \left( \frac{|u(x)-u(y)|^q}{|x-y|^{n+sq}} + 
 a(x,y) \frac{|u(x)-u(y)|^p}{|x-y|^{n+sp}} \right)\,dxdy
$$
and we have the following result.

\begin{prop}\label{prop. exam1}
Let $u\in W^{1,p}(\R^n)\cap W^{1,q}(\R^n)$. Then,
$$
\lim_{s\uparrow 1} (1-s)\mathcal{J}_{s,p,q}(u)= \mathcal{K}_{n,q} \int_{\R^n} |\nabla u(x)|^q \,dx + \mathcal{K}_{n,p}  \int_{\R^n} a(x,x)   |\nabla u(x)|^p\,dx
$$
where
$$
\mathcal{K}_{n,\kappa}:=\frac{1}{\kappa} \int_{\mathbb{S}^{n-1}}  |w_n|^\kappa \,dS_w, \qquad \kappa>1.
$$
Conversely, if $u\in L^p(\R^n)\cap L^q(\R^n)$ is such that
$$
\liminf_{s\uparrow 1} (1-s)\mathcal{J}_{s,p,q}(u)<\infty,
$$
then $u\in W^{1,p}(\R^n)\cap W^{1,q}(\R^n)$.
\end{prop}

\begin{proof}
From Theorem \ref{asym. behavior} the result holds for any $u\in C_c^2(\R^n)$. By the boundedness of $a$ and \cite[Theorem 1]{BBM}, 
\begin{align*}
\iint_{\R^n\times\mathbb{R}^n}a(x,y)\frac{|u(x)-u(y)|^p}{|x-y|^{s+sp}}\,dy &\leq 
a_+ \iint_{\R^n\times\mathbb{R}^n}\frac{|u(x)-u(y)|^p}{|x-y|^{s+sp}}\,dy\\
&\leq \frac{a_+ n \omega_n}{p}\left(\frac{1}{1-s} \|\nabla u\|_p^p + \frac{2^p}{s} \|u\|_p^p\right).
\end{align*}
and
$$
\iint_{\R^n\times\mathbb{R}^n}\frac{|u(x)-u(y)|^q}{|x-y|^{n+sq}}\,dy \leq  \frac{n\omega_n}{q}\left(\frac{1}{1-s} \|\nabla u\|_q^q + \frac{2^q}{s} \|u\|_q^q\right).
$$
Therefore, arguing as in \cite[Theorem 2]{BBM}, the result is extended to an arbitrary $u\in W^{1,p}(\R^n)\cap W^{1,q}(\R^n)$, and it holds that if
$$
\liminf_{s\uparrow 1} (1-s)\mathcal{J}_{s,p,q}(u)<\infty,
$$
then $u\in W^{1,p}(\R^n)\cap W^{1,q}(\R^n)$.
\end{proof}

\subsection{Logarithmic perturbations of powers}

Let $a$ be as in the previous subsection. We consider the generalized Young function  $G(x,y,t)=a(x,y)t^p (\log^+(t)+1)$, $t\geq 0$, where $p>1$ and $\log^+ (t):=\max\{0,\log t\}$.  In this case, 
$$
\J(u)=\iint_{\R^n\times\R^n} a(x,y)\frac{|u(x)-u(y)|^p}{|x-y|^{n+sp}} \left(\log^+\left(\frac{|u(x)-u(y)|}{|x-y|^s}\right)+1\right) \,dxdy
$$
and we have the following result:

\begin{prop}
Let  $u\in W^{1,\overline{G}}(\R^n)$. Then it holds that
$$
\lim_{s\uparrow 1} (1-s) \mathcal{J}_{s,G}(u) = \int_{\R^n} H_0(x,|\nabla u(x)|)\,dx,
$$
where 
\begin{align*}
H_0(x,t)=
\begin{cases}
\displaystyle a(x,x) t^p \mathcal{K}_{n,p}, &\text{ when }\, t|w_n| \leq 1,\\
\displaystyle a(x,x) t^p \left[ \mathcal{K}_{n,p} \left( \frac{p-1}{p} + \log t\right)+ \mathcal{K}_{\log, n,p}\right] + \frac{a(x,x)}{p^2}, &\text{ when }\, t|w_n|> 1,
\end{cases}
\end{align*}
being $w_n$ the $n-$th coordinate of any point in $\mathbb{S}^{n-1}$,
$$
\mathcal{K}_{n,p}=\frac{1}{p}\int_{\mathbb{S}^{n-1}}|w_n|^p\,dS_w \qquad \mbox{and} \qquad
\mathcal{K}_{\log,n,p}=\frac{1}{p}\int_{\mathbb{S}^{n-1}}|w_n|^p \log |w_n|\,dS_w.
$$
Conversely, if $u\in L^{\overline{G}}(\R^n)$ is such that
$$
\liminf_{s\uparrow 1} (1-s)\mathcal{J}_{s,G}(u)<\infty,
$$
then $u\in W^{1,\overline{G}}(\R^n)$.
\end{prop}
\begin{proof}
In this case, we can split the following integral as
\begin{align*}
\int_0^1 G(x,x,t|w_n|r)\frac{dr}{r}&= 
a(x,x) t^p |w_n|^p \left(\int_0^1 r^{p-1}\,dr + \int_0^1 r^{p-1} \log^+( t |w_n|r)\,dr\right).
\end{align*}
Define $t_*:= \frac{1}{t|w_n|}$. If $t_*\geq 1$,  then  $r\leq t_*$ for all $r\in (0,1)$ and in this case $\log^+(t|w_n|r)=0$, giving that
$$
\int_0^1 G(x,x,t|w_n|r)\frac{dr}{r} = a(x,x) \frac{t^p}{p} |w_n|^p .
$$
When $t_*<1$, we have
$$
\int_0^1 r^{p-1} \log^+( t |w_n|r)\,dr =\int_{t_*}^1 r^{p-1} \log( t |w_n|r)\,dr  =
\frac{1}{p^2}\left( p \log(t|w_n|) + \frac{1}{(t|w_n|)^p}-1 \right),
$$
which implies
$$
\int_0^1 G(x,x,t|w_n|r)\frac{dr}{r} = a(x,x) \frac{t^p}{p} |w_n|^p \left(\frac{p-1}{p} + \frac{1}{p(t|w_n|)^p} + \log(t|w_n|) \right),
$$
and the expression of $H_0(x,t)$ follows just by integrating the variable $w$ in $\mathbb{S}^{n-1}$.

Now, from Theorem \ref{asym. behavior} the result holds for any $u\in C_c^2(\R^n)$. On the other hand, by the boundedness of $a$, we have that
\begin{align*}
\mathcal{J}_{s,G}(u) &\leq 
a_+ \iint_{\R^n\times\mathbb{R}^n}A\left( \frac{|u(x)-u(y)|}{|x-y|^s}\right)\frac{dxdy}{|x-y|^n},
\end{align*}
where $A\colon [0,\infty)\to[0,\infty)$ is the Young function given by $A(t)=t^p(\log^+(t)+1)$ with $p>1$. Then,  arguing as in the last part of the proof of \cite[Theorem 4.1]{FBS}, the result holds for any $u\in W^{1,\overline{G}}(\R^n)$ (space which is equal to  $W^{1,A}(\R^n)$, since $A$ and $\overline G$ are equivalent Young functions), and it holds that if
$$
\liminf_{s\uparrow 1} (1-s)\mathcal{J}_{s,G}(u)<\infty,
$$
then $u\in W^{1,A}(\R^n)$.
\end{proof}

\subsection{Spaces with variable exponent}
Given a continuous function in the second variable $p\colon \R^n\times\R^n \to \R$ such that $1<p^-\leq p(x,y) \leq p^+ <\infty$ for all $x,y\in \R^n$, and a function $a$ as in the previous example, consider the generalized Young function $G(x,y,t)=a(x,y)t^{p(x,y)}$. In this case,
$$
\J(u):=\mathcal{J}_{s,p}(u)=\iint_{\R^n\times\R^n} a(x,y) \frac{|u(x)-u(y)|^{p(x,y)}}{|x-y|^{n+sp(x,y)}} \,dxdy.
$$

\begin{prop} \label{Wp+p-}
Let $u\in W^{1,p^+}(\R^n)\cap W^{1,p^-}(\R^n)$. Then it holds that
$$
\lim_{s\uparrow 1} (1-s) \mathcal{J}_{s,p}(u) = \int_{\R^n} H_0(x,|\nabla u(x)|)\,dx,
$$
where
$$
H_0(x,t)=K_{n,p} t ^{p(x,x)}, \quad \text{ with }\quad  K_{n,p}=\frac{a(x,x)}{p(x,x)}\int_{\mathbb{S}^{n-1}} |w_n|^{p(x,x)} \,dS_w.
$$
\end{prop}

\begin{proof}
In this case the expression of $H_0$ is immediate. From  Theorem \ref{asym. behavior} the limit holds for any $u\in C_c^2(\R^n)$. Due to the assumptions on $p$, one has that $a_-\min\{t^{p^+},t^{p^-}\} \leq G(x,y,t) \leq a_+\max\{t^{p^+},t^{p^-}\}$ for any $t\geq 0$ and $x,y\in \R^n$. Then proceeding as in \cite[Theorem 1]{BBM} (or \cite[Corollary 4]{K}), the limit holds for any $u\in W^{1,p^+}(\R^n)\cap W^{1,p^-}(\R^n)$.
\end{proof}

\section{Anisotropic $s-$H\"older quotients} \label{sec5}
Fractional anisotropic spaces in which in each coordinate direction the functions have different fractional regularity and different integrability have been considered recently, see \cite{CKW,DB}.

In this section we consider a family of functionals in which the $s-$H\"older quotients depend on only one direction. Given a generalized Young function $G\colon \R^n\times\R^n \times [0,\infty) \to \R$, $s\in (0,1)$ and $k\in\{1,\ldots, n\}$, we consider the energy functional
$$
\mathcal{J}_{s,G}^k (u):=\int_{\R^n} \int_{\R} G(x,x-he_k,|D_s^k u(x,h)|)\,\frac{dhdx}{|h|},
$$
where $e_k$ is the $k-$th canonical vector in $\R^n$ and the $s-$H\"older quotient $D_s^k u$ in the direction $e_k$ is defined as
$$
D_s^k u(x,h):=\frac{u(x-he_k)-u(x)}{|h|^s}.
$$
These functionals naturally define the Sobolev-like spaces $W^{s,G}_k(\R^n)$ as
$$
W^{s,G}_k(\R^n)=\left\{ u \in L^{\overline G(\cdot)}(\R^n) \colon \mathcal{J}_{s,G}^k(u) <\infty\right\}.
$$

With the same technique than in Theorem \ref{asym. behavior}, we prove a BBM result for smooth functions.

\begin{thm}\label{thm-anisotropic}
Let $u\in C^2_c(\R^n)$ and  $k\in\{1,\ldots, n\}$. Then, 
\begin{align}\label{eq. asym. behavior}
	\lim_{s\uparrow 1} (1-s) \J^k(u) = \int_{\R^n} H_0\left(x,\left|\frac{\partial u(x)}{\partial x_k} \right|\right)\,dx,
\end{align}
where 
$$
H_0(x,t)=2 \int_0^1 G(x,x,tr) \frac{dr}{r} .
$$
\end{thm}

\begin{proof}
Let $u\in C^2_c(\R^n)$ and let $x$ be fixed. Without loss of generality we can assume $k=1$.  Proceeding similarly as in the proof of Theorem \ref{asym. behavior}, we can obtain the following
\begin{align*}
\lim_{s\uparrow 1} (1-s)\int_{\R}& G(x,x-he_1,|D_s^1 u(x,h)|)\,\frac{dh}{|h|}\\ 
&=  
\lim_{s\uparrow 1} (1-s)\int_{|h|<1} G\left(x,x-he_1,\left| \nabla u(x)\cdot e_1 h|h|^{-s} \right| \right)\, \frac{dh}{|h|}\\
&=  
\lim_{s\uparrow 1} (1-s)\int_{|h|<1} G\left(x,x-he_1,|\nabla u(x)\cdot e_1||h|^{1-s} \right)\, \frac{dh}{|h|}\\
&=  
\lim_{s\uparrow 1} (1-s)\int_{|h|<1} G\left(x,x-he_1,\left|\frac{\partial u(x)}{\partial x_1}\right||h|^{1-s} \right)\, \frac{dh}{|h|}\\
&=:\lim_{s\uparrow 1} (1-s) I_s(x).
\end{align*}

Now, observe that, similarly as in the proof of  Lemma \ref{conv. Young function}, performing the change of variables $h^{1-s}=\rho$ we get that $(1-s)\frac{dh}{h}=\frac{d\rho}{\rho}$ which yields
\begin{align*}
I_s(x) &=
\int_0^1 G\left(x,x-he_1,\left|\frac{\partial u(x)}{\partial x_1}\right| h^{1-s} \right) \frac{dh}{h} -\int_{-1}^0 G\left(x,x-he_1,\left|\frac{\partial u(x)}{\partial x_1}\right| (-h)^{1-s} \right) \frac{dh}{h}  \\
&=\int_0^1 \left(G\left(x,x-he_1,\left|\frac{\partial u(x)}{\partial x_1}\right| h^{1-s} \right) 
+ G\left(x,x+he_1,\left|\frac{\partial u(x)}{\partial x_1}\right| h^{1-s} \right) 
 \right)\frac{dh}{h}\\
&=\frac{1}{1-s}\int_0^1 \left(G\left(x,x-\rho^\frac{1}{1-s}e_1,\left|\frac{\partial u(x)}{\partial x_1}\right|\rho\right)+ 
G\left(x,x+\rho^\frac{1}{1-s}e_1,\left|\frac{\partial u(x)}{\partial x_1}\right|\rho\right) \right)\frac{d\rho}{\rho}.
\end{align*}
Then, since we assume $G$ continuous in the second parameter and $\rho\in(0,1)$, 
$$
\lim_{s\uparrow 1} (1-s) I_s(x) = 2 \int_0^1 G\left(x,x,\left|\frac{\partial u(x)}{\partial x_1}\right|\rho \right) \frac{d\rho}{\rho} .
$$
Thus, arguing as in the last part of the proof of Theorem \ref{asym. behavior}, the limit \eqref{eq. asym. behavior} holds.  

Finally, as in Proposition \ref{prop.H0}, the function $H_0(x,t)=2\int_0^1 G(x,x,tr)\frac{dr}{r}$, up to constant, is comparable with $\overline G(x,t)$.
\end{proof}

In the case in which there is no dependence on the spatial variables and a the $s-$H\"older quotient has a power behavior, i.e., 
$$
G(x,y,t) = t^p, \quad p>1, \, s\in (0,1),
$$
for each $k\in \{1,\ldots, n\}$, the limit function $H_0$ is easily computed as
$$
H_0(x,t)= 2t^p \int_0^1  r^{p-1}\,dr = \frac{2}{p}t^p,
$$
which implies that 
$$
\lim_{s\uparrow 1}(1-s) \int_{\R^n} \int_{\R} \frac{|u(x-he_k)-u(x)|^p}{|h|^{1+sp}}\,dhdx = \frac{2}{p} \int_{\R^n} \left| \frac{\partial u(x)}{\partial x_k}\right|^p\,dx
$$
holds for any $u\in C^2_c(\R^n)$. Thus, arguing as in Proposition \ref{Wp+p-}, it holds for any $u\in W^{s,p}_k(\R^n)$. This recovers the limit result of \cite{DB}.

\begin{rem}
When the dependence of the variables $x$ and $y$ is removed in the anisotropic energy, much more information can be obtained. Indeed, in the case in which
$$
\mathcal{J}_{s,G}^k (u):=\int_{\R^n} \int_{\R} G(|D_s^k u(x,y)|)\frac{dh}{|h|}\,dx
$$
following the last part of the proof of \cite[Theorem 4.1]{FBS}, it is not too hard to see that, for any $u\in L^G(\R^n)$ and $k\in\{1,\ldots, n\}$, 
$$
\lim_{s\uparrow 1} (1-s) \J^k(u) = 2 \int_{\R^n} \int_0^1 G\left(\left|\frac{\partial u(x)}{\partial x_k} \right| r \right) \frac{dr}{r} \,dx.
$$
Moreover, it is possible in this case to go further and prove a BBM type result for sequences of functions. A close inspection of \cite[Theorem 5.2]{BBM} reveals that in this case, if $0\leq s_j\uparrow 1$ and $\{u_j\}_{j\in \N}\subset L^G(\R^n)$ is such that
$$
\sup_{j\in \N}\left((1-s_j)\J^k(u_j) + \int_{\R^n} G(u_j)\,dx\right)<\infty, 
$$
then there exists $u\in L^G(\R^n)$ and a subsequence $\{u_{\ell}\}_{\ell\in\N} \subset \{u_j\}_{j\in\N}$ such that $u_{\ell}\to u$ in $L^G_{loc}(\R^n)$. Moreover, $u\in W^{1,G}(\R^n)$ and 
$$
2\int_{\R^n} \int_0^1 G\left(\left| \frac{\partial u(x)}{\partial x_k}\right|r\right)\,\frac{dr}{r}dx \leq \liminf_{\ell\to\infty} (1-s_\ell)\J^k(u_\ell).
$$
\end{rem}

\medskip
 \begin{acknowledgement}
 	The first and third authors were partially supported by CNPq with grants 304699/2021-7 and 316643/2021-1, respectively. The second author was partially supported by CAPES. The fourth author was partially supported y CONICET under grant PIP 11220150100032CO, by ANPCyT under grants PICT 2016-1022 and PICT
 	2019-3530 and by the University of Buenos Aires under grant 20020170100445BA.
 \end{acknowledgement}

\bigskip
\end{document}